\documentclass[a4paper,10pt]{siamltex}
\usepackage[utf8x]{inputenc}
\usepackage[english]{babel}
\usepackage[a4paper]{geometry}
\usepackage{pgfplots}
\usepackage{amsmath}
\usepackage{amssymb}
\usepackage{amsfonts}
\usepackage{listings}
\usepackage{comment}

\usepackage{algorithm}
\usepackage{algorithmicx}
\usepackage{algpseudocode}

  \newcommand{\gr}{
   \raisebox{-3pt}{
    \makebox[-2pt] {\includegraphics[height=12pt]{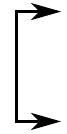}}
    }
  }
  
  \newcommand{\redgr}{
   \raisebox{-3pt}{
    \makebox[-2pt] {\includegraphics[height=12pt]{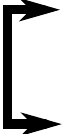}}
    }  
  }
  
  \newcommand{\ascendinggr}{
  \raisebox{6pt}  { \ensuremath {\begin{array}{lll}
          \\
           \\
          \gr \\
          &\gr \\
          &&\gr \\
        \end{array}}}
  }
  
    \newcommand{\ascendinggrhats}{
    \raisebox{6pt}  { \ensuremath {\begin{array}{lll}
            \\
             \\
            \hat\gr \\
            &\hat\gr \\
            &&\hat\gr \\
          \end{array}}}
    }
    
    \newcommand{\descendinggrhats}{
    \raisebox{6pt}  { \ensuremath {\begin{array}{llll}
            \\
             \\
            &&\hat\gr \\
            &\hat\gr \\
            \hat\gr \\
          \end{array}}}
    }

      \newcommand{\descendinggrfull}{
        \raisebox{6pt}  { \ensuremath {\begin{array}{llll}
                \\
                &&&\gr \\
                &&\gr \\
                &\gr \\
                \gr \\
              \end{array}}}
        }
        
      \newcommand{\descendinggr}{
        \raisebox{6pt}  { \ensuremath {\begin{array}{lll}
                \\
                \\
                &&\gr \\
                &\gr \\
                \gr \\
              \end{array}}}
        }

\newtheorem{cor}{Corollary}
\newtheorem{defi}{Definition}
\newtheorem{rem}{Remark}

\DeclareMathOperator{\tril}{tril}
\DeclareMathOperator{\triu}{triu}

\newcommand{\conj}[1]{\overline{#1}}

\newcommand{\qsh}[1]{\mathcal{QSH}_{#1}}

\newcommand{\gv}[1]{\mathrm{GV}(#1)}

\renewcommand{\hat}{\widehat}
\renewcommand{\tilde}{\widetilde}
\renewcommand{\leq}{\leqslant}
\renewcommand{\geq}{\geqslant}

\title{Quasiseparable Hessenberg reduction of real diagonal plus low rank matrices  and applications\thanks{Work supported by Gruppo Nazionale di Calcolo Scientifico (GNCS) of INdAM}}
\author{Dario A. Bini
\thanks{Dipartimento di Matematica, Universit\`a di Pisa,
Largo Bruno Pontecorvo 5, 56127 Pisa ({\tt bini@dm.unipi.it})}
\and
Leonardo Robol
\thanks{Scuola Normale Superiore, Piazza dei Cavalieri, Pisa
({\tt leonardo.robol@sns.it})}
}

\begin{document}

  \maketitle

  \begin{abstract}
    We present a novel algorithm to perform the Hessenberg reduction
    of an $n\times n$ matrix $A$ of the form $A = D + UV^*$ where $D$
    is diagonal with real entries and $U$ and $V$ are $n\times k$
    matrices with $k\le n$. The algorithm has a cost of $O(n^2k)$
    arithmetic operations and is based on the quasiseparable matrix
    technology. Applications are shown to solving polynomial
    eigenvalue problems and some numerical experiments are reported
    in order to analyze the stability of the approach. 
  \end{abstract}
  
  \smallskip
  
  
  \smallskip
  

  \section{Introduction}
Reducing an $n\times n$ matrix to upper Hessenberg form by a unitary
similarity is a fundamental step at the basis of most numerical
methods for computing matrix eigenvalues. For a general matrix, this
reduction has a cost of $O(n^3)$ arithmetic operations (ops), while
for matrices having additional structures this cost can be
lowered. This happens, for instance, for the class of quasiseparable
matrices. We say that a matrix is $(k1,k2)$-quasiseparable if its
submatrices contained in the strictly upper triangular part have rank at most
$k_2$ and those contained in the strictly lower triangular part have rank at
most $k_1$.  For simplicity, if $k_1=k_2=k$ we say that $A$ is
$k$-quasiseparable.  Quasiseparable matrices have recently received
much attention; for properties of this matrix class we refer the
reader to the recent books \cite{eidelman:book1}, \cite{eidelman:book2}, 
\cite{vanbarel:book2}, \cite{vanbarel:book1}.

In the papers \cite{eidelman-luca-hess} and \cite{vanbarel-hess},
algorithms are provided to reduce a $k$-quasiseparable matrix $A$ to
upper Hessenberg form $H$ via a sequence of unitary
transformations. If $A$ satisfies some additional hypothesis then the Hessenberg matrix obtained this way is still
quasiseparable, and the cost of this reduction is $O(n^2k^\alpha)$
ops, where $\alpha$ is strictly greater than 1,  in particular
$\alpha = 3$ for the algorithm of \cite{eidelman-luca-hess}. The
advantage of this property is that one can apply the shifted QR
iteration to the matrix $H$ at the cost of $O(nk^2)$ ops per step,
instead of $O(n^2)$, by exploiting both the Hessenberg and the
quasiseparable structure of $H$, see for instance \cite{ego05}, \cite{vvm05}. This way,
the computation of the eigenvalues of $A$ has a lower cost with
respect to the case of a general matrix. More specifically, assuming
that the number of QR iterations is $O(n)$, the overall cost for
computing all the eigenvalues turns to
$O(n^2k^2+n^2k^\alpha)$. Clearly, if $\alpha\ge 3$, the advantage
of this quasiseparable reduction to Hessenberg form can be appreciated
only if $k\ll n$. In particular, if the value of $k$ is of the order
of $n^{\frac13}$ these algorithms have cost $O(n^3)$ as for a general
unstructured matrix.

In this paper we consider the case of a $k$-quasiseparable matrix
which can be represented as
    \begin{equation}\label{eq:d+lr}
      A = D + UV^*, \quad U, V \in \mathbb{C}^{n \times k},
    \end{equation}
where $D$ is a diagonal matrix with real entries and $U,V\in\mathbb
C^{n\times k}$ and $V^*$ is the Hermitian transpose of $V$. Matrices
of this kind are encountered in the linearization of matrix
polynomials obtained by the generalization of the Smith companion form
\cite{smith}, \cite{secular-linearization}.  They have been also used
in the package MPSolve v. 3.1.4 for the solution of polynomial and
secular equations up to any desired precision \cite{br:jcam}.

We prove that if $H=QAQ^*$ is the Hessenberg form of $A$, with $Q$
unitary, then $H$ is $(1,2k-1)$-quasiseparable, moreover, we provide
an algorithm for computing $H$ with $O(n^2k)$ ops. This algorithm
substantially improves the algorithms of \cite{eidelman-luca-hess} and
\cite{vanbarel-hess} whatever is the value of $k$. Moreover, for an
unstructured matrix where $k= n$, the cost of our
algorithm amounts to $O(n^3)$, that is the same asymptotic cost of the
Hessenberg reduction for a general matrix.

An immediate consequence of this algorithm is that the cost for
computing all the eigenvalues of $A$ by means of the shifted QR
iteration applied to $H$ turns to $O(n^2k)$ with an acceleration by
a factor of $O(k^{\alpha-1})$ with respect to the algorithms of
\cite{eidelman-luca-hess} and \cite{vanbarel-hess}.

Another application of this result concerns the solution of polynomial
eigenvalue problems and is the main motivation of this work.  Consider
the matrix polynomial $P(x) = \sum_{i = 0}^n P_i x^i$, $P_i \in
\mathbb{C}^{m \times m}$, where for simplicity we assume $P_n = I$.
The polynomial eigenvalue problem consists in computing the solution
of the equation $\det P(x)=0$ and, if needed, the nonzero vectors $v$
such that $P(x)v=0$.

 The usual strategy adopted in this case is to employ a linearization
 $L(x) = xI- A$ of the matrix polynomial, such that the linear
 eigenvalue problem $L(x)w = 0$ is equivalent to the original one
 $P(x)v = 0$. The matrix $A$, also called companion matrix of $P(x)$,
 is of size $mn \times mn$. Many companion matrices do have a
 $k$-quasiseparable structure where $k=m$ and in the case of the Smith
 companion \cite{smith}, generalized to matrix polynomials in
 \cite{secular-linearization}, the quasiseparable structure takes the
 desired form $A=D + UV^*$. In this case, the cost of our
 algorithm to reduce $A$ into quasiseparable Hessenberg form turns to
 $O(nm^2)$ whereas the algorithms of \cite{eidelman-luca-hess} and
 \cite{vanbarel-hess} of cost $O(nm^{\alpha+1})$, where $\alpha > 1$,
 would be unpractical.

    This paper is divided in 6 sections. Besides the introduction, in
    Section \ref{sec:rank-conservation} we introduce some preliminary
    results, concerning quasiseparable matrices, which are needed to
    design the algorithm for the Hessenberg reduction.  
In Section~\ref{sec:rthf} we recall the
general algorithm for reducing a matrix into Hessenberg form by means
of Givens rotations and prove the main result, expressed in Theorem
\ref{thm:rankstructure}, on the conservation of the quasiseparable
structure at all the steps of the algorithm.  
In Section~\ref{sec:qsr} we recall and elaborate the definition of
Givens vector representation of a symmetric $k$-quasiseparable
matrix. Then we rephrase Theorem \ref{thm:rankstructure} in terms of
Givens vector representations.
Section~\ref{sec:reduction-algorithm} deals with algorithmic
    issues: the fast algorithm for the Hessenberg reduction is presented
in detail. In
    Section~\ref{sec:numerical-experiments} we present some numerical
    experiments and show that the CPU time needed in our tests confirms 
the complexity bound $O(n^2k)$.
Finally the last Section \ref{sec:appl} briefly shows an application of these results to numerically computing the eigenvalues of a matrix polynomial.
  
  \section{Preliminary tools} \label{sec:rank-conservation}

Throughout the paper the matrix $A$ has the form \eqref{eq:d+lr} and
$H=QAQ^*$ is in upper Hessenberg form, where $Q$ is unitary, i.e., $Q^*Q=I$.
Here we recall the notations and definitions used in this paper, which
mainly comply with the ones used in \cite{vanbarel:book1}, together with the
main results concerning quasiseparable matrices.
      
       \begin{defi}
      A complex matrix $A$ is \emph{lower-quasiseparable}
      (resp. \emph{upper-quasiseparable}) \emph{of rank $k$} if every
      submatrix contained in the strictly lower (resp. upper)
      triangular part of $A$ has rank at most $k$. If $A$ is $k_l$
      lower quasiseparable and $k_u$ upper quasiseparable we say that
      $A$ is $(k_l,k_u)$-quasiseparable. If $k = k_l = k_u$ we say
      that $A$ is $k$-quasiseparable. Moreover, we denote $\qsh{k}^n$
      the set of $n\times n$ Hermitian $k$-quasiseparable matrices
      with entries in $\mathbb C$. We will sometimes omit the superscript $n$
      and simply write $\qsh{k}$ when the dimension is clear from the context. In general, we say that $A$ is
      quasiseparable if it is $(k_l,k_u)$-quasiseparable for some
      nontrivial $k_l,k_u$.
      \end{defi}
       
 The definition of quasiseparable matrix can be easily
 expressed by using the MATLAB notation. In fact, $A$ is
 lower-quasiseparable of rank $k$ if and only if
         \[
         \rank(A[i+1:n,1:i]) \leq k\qquad \hbox{for }i=1,\ldots,n-1,
         \]
where $A[i_1:i_2,j_1:j_2]$ denotes the submatrix of $A$ formed by the entries $a_{i,j}$ for $i=i_1,\ldots,i_2$, $j=j_1,\ldots,j_2$.

 Given a vector
$v=(v_i)\in\mathbb C^2$, denote by $G=G(v_1,v_2)$ a $2\times 2$ Givens
rotation
      \[
        G = \begin{bmatrix} 
         c &  s \\
         -\conj s & c \\
         \end{bmatrix},\quad c\in\mathbb R,\quad |s|^2+c^2=1,
      \]
such that $Gv=\alpha e_1$, where $|\alpha|=\sqrt{|v_1|^2+|v_2|^2}$.
More generally, denote $G_i=I_{i-1}\oplus G\oplus I_{n-i-1}$ the
$n\times n$ matrix which applies the Givens rotation
$G=\left[\begin{smallmatrix}c_i&s_i\\-\conj s_i&
    c_i\end{smallmatrix}\right]$ to the components $i$ and $i+1$ of a
vector, where  $I_m$ denotes the identity matrix of size $m$ and $A\oplus B$ denotes the block
diagonal matrix
$\left[\begin{smallmatrix}A&0\\0&B\end{smallmatrix}\right]$.
With abuse of terminology, we will call $G_i$ Givens
rotations as well.

The following well-known result \cite{vanbarel:book1} will be used in
our analysis:
    
    \begin{lemma} \label{lem:unitaryQ}
      Let $Q$ be a unitary Hessenberg matrix. Then
      \begin{itemize}
        \item $Q$ is a $(1,1)$-quasiseparable matrix. 
        \item $Q$ can be factorized as a product of $n-1$ Givens rotations
          $Q = G_{1} \dots G_{n-1}$.
      \end{itemize}
    \end{lemma}	   

In the following we use the operators $\tril(\cdot,\cdot)$ and
$\triu(\cdot,\cdot)$, coherently with the corresponding MATLAB functions, 
such that $L = \tril(A,k)$,  $U = \triu(A,k)$ where $A=(a_{i,j})$, $L=(\ell_{i,j})$, $U=(u_{i,j})$ and
\[
  \ell_{i,j} = \begin{cases}
    a_{i,j} & \text{if } i \geq j - k \\
    0 & \text{otherwise} \\
  \end{cases} \qquad  
  u_{i,j} = \begin{cases}
      a_{i,j} & \text{if } i \leq j - k \\
      0 & \text{otherwise} \\
    \end{cases}.
\]

\subsection{A useful operator}  
Another useful tool is the  operator $t:\mathbb C^{n\times n}\to\mathbb C^{n\times n}$ defined by
 \begin{equation}\label{eq:tk}
t(A)=\hbox{tril}(A,-1)+\hbox{triu}(A^*,1).
\end{equation}
Observe that if $A$ has rank $k$ then $t(A)$ is a Hermitian
$k$-quasiseparable matrix with zero diagonal entries.  In particular,
for $u,v\in\mathbb C^n$, the matrix $t(uv^*)$ is in $\qsh{1}^n$ and
its entries are independent of $u_1$ and of $v_n$.  More generally,
for $U,V\in\mathbb C^{n\times k}$, the matrix $t(UV^*)$ is in
$\qsh{k}^n$ and its entries are independent of the first row $U[1,:]$
of $U$ and of the last row $V[n,:]$ of $V$.
Observe also that $t(A)$ is independent of the upper triangular part of $A$.

The following properties can be verified by a direct inspection
\begin{equation}\label{eq:propt}
\begin{split}
&t(A+B)=t(A)+t(B),\hbox{ for any }A,B,\in\mathbb C^{n\times n},\\
&t(\alpha A)=\alpha t(A),\hbox{ for any }\alpha\in\mathbb R,\\
&t(DAD^*) = Dt(A)D^*, \hbox { for any } D \hbox{ diagonal matrix},\\
\end{split}
\end{equation}
moreover,
\begin{equation}\label{eq:proptbis}
t\left({\small \begin{bmatrix}A_{1,1}&A_{1,2}\\ A_{2,1}&A_{2,2}\end{bmatrix}}\right)=
\begin{bmatrix}t(A_{1,1})&A_{2,1}^*\\ A_{2,1}&t(A_{2,2})\end{bmatrix},
\end{equation}
where $A_{1,1}$ and $A_{2,2}$ are square matrices. We also have
\begin{equation}\label{eq:propt1}
t(A)=A-\diag(a_{1,1},\ldots,a_{n,n}),\qquad \hbox{for any $A$ such that }A=A^*.
\end{equation}

We analyze some properties of the residual matrix $R=t(SAS^*)-St(A)S^*$, for
$S$ being a unitary upper Hessenberg matrix, which will be used to prove the main
result in the next section.  We start with a couple of technical
lemmas.

\begin{lemma}\label{lem:t1}
Let $Z \in\mathbb C^{k\times k}$ and set $S=Z\oplus I_{n-k}$ where
$n>k$. Then for any $A\in\mathbb C^{n\times n}$  it holds that
\[
t(SAS^*)-St(A)S^*=W\oplus 0_{n-k},
\]
for some $W\in\mathbb C^{k\times k}$ where $0_{n-k}$ is the null matrix of size $n-k$.
Similarly, for $S=I_{n-k}\oplus Z$ it holds that  $t(SAS^*)-St(A)S^*=0_{n-k}\oplus W'$, for some 
$W'\in\mathbb C^{k\times k}$. The same properties hold if $I_{n-k}$ is replaced by a diagonal matrix $D_{n-k}$.
\end{lemma} 
\begin{proof}
Concerning the first part,  partition $A$ as 
$
A=\left[\begin{smallmatrix}A_{1,1}&A_{1,2}\\ A_{2,1}& A_{2,2}
\end{smallmatrix}\right]
$
where $A_{1,1}\in\mathbb C^{k\times k}$, so that
$
SAS^*=\left[\begin{smallmatrix}ZA_{1,1}Z^*&ZA_{1,2}\\ A_{2,1}Z^*& A_{2,2}
\end{smallmatrix}\right].~
$
In view of \eqref{eq:proptbis} we have
\begin{equation}\label{eq:eq1}
t(SAS^*)=\begin{bmatrix}t(ZA_{1,1}Z^*)&ZA_{2,1}^*\\ A_{2,1}Z^*& t(A_{2,2})
\end{bmatrix}.
\end{equation}
On the other hand,
\begin{equation}\label{eq:eq2}
St(A)S^*=S\begin{bmatrix}t(A_{1,1})&A_{2,1}^*\\ A_{2,1}& t(A_{2,2})
\end{bmatrix}S^*=\begin{bmatrix}Zt(A_{1,1})Z^*&ZA_{2,1}^*\\ A_{2,1}Z^*& t(A_{2,2})
\end{bmatrix}.
\end{equation}
So that, from \eqref{eq:eq1} and \eqref{eq:eq2} we get
$t(SAS^*)-St(A)S^*=W\oplus 0_{n-k}$, with
$W=t(ZA_{1,1}Z^*)-Zt(A_{1,1})Z^*$.
The second part can be proved similarly.
Finally, if $I_{n-k}$ is replaced by the diagonal matrix $D_{n-k}$ the same properties hold since for a diagonal matrix $D$ one has $t(DAD^*)-Dt(A)D^*=0$ in view of \eqref{eq:propt}.
\end{proof}

\begin{lemma}\label{lem:t2}
Let $S=(Z\oplus I_{n-2})(1\oplus\hat S)$ where $Z\in\mathbb C^{2\times 2}$, $\hat S\in\mathbb C^{(n-1)\times(n-1)}$.
Then for any matrix $A$ partitioned as $A={\scriptsize \begin{bmatrix}a_{1,1}&u^*\\ v& \hat A\end{bmatrix}}\in\mathbb C^{n\times n}$,
where $\hat A\in\mathbb C^{(n-1)\times(n-1)}$  it holds that
\[
t(SAS^*)-St(A)S^*=W\oplus 0_{n-2}+(Z\oplus I_{n-2})(0\oplus(t(\hat S\hat A\hat S^*)-\hat St(\hat A)\hat S^*)) (Z^*\oplus I_{n-2}).
\]
for some $W\in\mathbb C^{2\times 2}$.
\end{lemma}
\begin{proof}
Set $B=(1\oplus \hat S)A(1\oplus \hat S^*)$ then by Lemma \ref{lem:t1}
\[
t(SAS^*)=t((Z\oplus I_{n-2})B(Z^*\oplus I_{n-2})=W\oplus 0_{n-2}+(Z\oplus I_{n-2})t(B)(Z^*\oplus I_{n-2}).
\]
On the other hand
\[
St(A)S^*=(Z\oplus I_{n-2}) (1\oplus \hat S) t(A)(1\oplus \hat S^*)(Z^*\oplus I_{n-2}).
\]
Thus
\[
t(SAS^*)-St(A)S^*=W\oplus 0_{n-2}+ (Z\oplus I_{n-2})E(Z^*\oplus I_{n-2}), 
\]
where $E=t(B)-(1\oplus \hat S) t(A)(1\oplus \hat S)$. Now, since $B=(1\oplus \hat S)A(1\oplus \hat S^*)$, in view of \eqref{eq:proptbis} we have
\[
B=\begin{bmatrix} a_{1,1}&u^*\hat S^*\\ \hat Sv&\hat S\hat A\hat S^* \end{bmatrix},\quad
 t(B)=\begin{bmatrix} 0&v^*\hat S^*\\ \hat Sv&t(\hat S\hat A\hat S^*)\end{bmatrix}.
\]
A similar analysis shows that
\[
(1\oplus\hat S)t(A)(1\oplus\hat S^*)=\begin{bmatrix}0&v^*\hat S^*\\ \hat S v&\hat S t(\hat A)\hat S^*  \end{bmatrix}.
\]
Thus we get
\[
t(SAS^*)-St(A)S^*=W\oplus 0_{n-2}+(Z\oplus I_{n-2})(0\oplus(t(\hat S\hat A\hat S^*)-St(\hat A)\hat S^*)) (Z^*\oplus I_{n-2}).
\]
\end{proof}
 
A consequence of the above results is expressed by
the following

\begin{theorem}\label{th:0}
Let $A\in\mathbb C ^{n\times n}$, and set $Q=G_h\cdots
G_{k}$ for $1\le h<k\le n-1$, where the
parameters $s_i,c_i$ defining $G_i$  are such that
$s_i\ne 0$, $i=h,\ldots,k$.
 Then $R_n:=t(QAQ^*)-Qt(A)Q^*=\diag(d)+t(ab^*)\in\qsh{1}^n$ for vectors
$a,b,d\in\mathbb C^n$, where $b$ is independent of $A$. 
More precisely,  $b_h=s_h\cdots s_{k}$, $b_i=c_{i-1}s_{i}\cdots s_{k}$, 
for $i=h+1,\ldots,k$, $b_i=a_i=d_i=0$ for $i<h$ and $i> k+1$.  In particular, if 
$h > 1$ then $R_ne_1=0$.
\end{theorem}
\begin{proof}
Clearly the matrix $Q$ has the form $I_{h-1}\oplus Z_{k-h+2}\oplus
I_{n-k-1}$ where $Z_{k-h+2}$ is a unitary Hessenberg matrix of size
$k-h+2$. In view of Lemma \ref{lem:t1}, we can write
$R_n=0_{h-1}\oplus R_{k-h+2}\oplus 0_{n-k-1}$  and this immediately
proves the last statement of the Theorem. Moreover, it follows that $a_i=b_i=d_i=0$
for $i=1,\ldots, h-1$ and for $i=k+2,\ldots,n$ so that it is
sufficient to prove the claim for $R_{k-h+2}$. Equivalently, we may
assume that $h=1$ and $k=n-1$ so that $s_i\ne 0$ for $i=1,\ldots,n-1$.
We prove that $R_n\in\qsh{1}^n$ by induction on $n$.  For $n=2$ it
holds that $R_2=\left[\begin{smallmatrix}0 & \overline{\alpha}
    \\\alpha &0 \end{smallmatrix}\right]$. This way one can choose
$a_2=\alpha/s_1$ and $b_1=s_1$.  For the inductive step, let $n>1$ and
observe that $Q$ can be factorized as $Q=(Z\oplus I_{n-2})(1\oplus
\hat S)$ for $Z=\left[\begin{smallmatrix}c &s \\ -\conj s
    &c \end{smallmatrix}\right]$, and $\hat S\in\mathbb
C^{(n-1)\times(n-1)}$, where for notational simplicity we set
$s=s_{1}$, $c=c_{1}$.  Applying Lemma \ref{lem:t2} yields
\[
R_n=W\oplus 0_{n-2}+(Z\oplus I_{n-2})(0\oplus R_{n-1})(Z^*\oplus I_{n-2})
\]
for
\[\begin{split}
W&=t(Z\left[\begin{smallmatrix}a_{1,1}&a_{1,2}\\ a_{2,1}&a_{2,2}\end{smallmatrix}\right]Z^*)
-Zt(\left[\begin{smallmatrix}a_{1,1}&a_{1,2}\\ a_{2,1}&a_{2,2}\end{smallmatrix}\right])Z^* \\
&=\left[ \begin{smallmatrix}
  -c (s a_{2,1} + \conj{s a_{2,1}}) & -s ( \conj c \conj a_{1,1} + s \conj a_{1,2} - \conj c \conj a_{2,2} -s \conj a_{2,1}) \\
  -\conj s ( c a_{1,1} + \conj s a_{1,2} -c a_{2,2} -\conj s a_{2,1}) & c (sa_{2,1} + \conj{(sa_{2,1})}) \\
\end{smallmatrix} \right],
\end{split}
\]
 where $R_{n-1}=t(\hat S\hat A\hat S^*)-\hat St(\hat A)\hat S^*$ and $\hat A$ is the trailing principal submatrix of $A$ of size $n-1$.
A direct inspection shows that
\begin{equation}\label{eq:tmp1}
R_n=W\oplus 0_{n-2}+
\left[\begin{array}{c|c}
|s|^2e_1^TR_{n-1}e_1&se_1^T R_{n-1}D\\ \hline 
\conj s DR_{n-1}e_1&DR_{n-1}D
\end{array}\right],\quad D=c \oplus I_{n-2}.
\end{equation}
  From the inductive assumption we may write that
$R_{n-1}=\diag(\hat d)+t(\hat a\hat b^*)$ for $\hat a,\hat b,\hat
d\in\mathbb C^{n-1}$, where $\hat b_1=s_{2}\cdots s_{n-1}\ne 0$, $\hat
b_i=c_{i} s_{i+1}\cdots s_{n-1}$.
So that
\eqref{eq:tmp1} turns into
{\small \[
R_n=
\left[\begin{array}{ccccc}
w_{1,1}&w_{1,2}\\
w_{2,1}&w_{2,2}\\
&&\\
&&&\\
&&&&
\end{array}\right]+
\left[\begin{array}{c|c|cccc}
|s|^2\hat d_1&*&*&\cdots&\ldots&*\\ \hline 
\hat d_1 c \conj s&c^2\hat d_1&*&\ldots&\ldots& *\\ \hline
\hat a_2\overline{\hat b}_1\conj s&\hat a_2\overline{\hat b}_1c&\hat d_2 &\ddots&\ddots&\vdots\\
\vdots&\vdots&\hat a_3\overline{\hat b}_2&\hat d_3&\ddots&\vdots\\
\vdots&\vdots&\vdots                    &\ddots  &\ddots&*\\
   \hat a_{n-1}\overline{\hat b}_1 \conj s&
        \hat a_{n-1}\overline{\hat b}_1c& \hat a_{n-1}\overline{\hat b}_2&
           \ldots&
               \hat a_{n-1}\overline{\hat b}_{n-2}&
                   \hat d_{n-1}
\end{array}\right],
\] }
where the upper triangular part, denoted with $*$, is determined by
symmetry.  Thus, it follows that $R_n=\diag(d)+t(a,b)$ where
$d_1=|s_{1}|^2\hat d_1+w_{1,1}$, $d_2=c_{1}^2\hat d_1+w_{2,2}$, $d_i=\hat
d_{i-1}$, for $i=3,\ldots,n$; moreover
\[
\begin{split}
&a_2=\frac{1}{\overline{\hat
  b}_1}(c\conj s\hat d_1 + w_{2,1}),\\
 &a_i=\hat a_{i-1}, \hbox{ for }i=3,\ldots,n,\\ 
&b_1=s_{1}\hat b_1,~~ b_2=c_{1}\hat b_1,\\
&b_i=\hat b_{i-1}, \hbox{ for }i=3,\ldots,n-1,
\end{split}
\]
 where the condition $s_i\ne
0$ implies that $b_1\ne 0$. This completes the proof.
\end{proof}

The property that $R=D+t(ab^*)$, where $b$ is independent of $A$
enables us to prove the following

\begin{cor}\label{cor:2}
In the hypotheses of Theorem \ref{th:0} we have
\[
R=t(QAQ^*)-Q(D+t(A))Q^*\in\qsh{1}^n
\]
for any diagonal matrix $D$ with real entries. Moreover, $R=\diag(d)+t(ab^*)$ for some vectors $a,b\in\mathbb C^n$ and $d\in\mathbb R^n$.  
\end{cor}
\begin{proof}
Without loss of generality we may assume that $s_i\ne 0$ for
$i=1,\ldots,n-1$.  In view of Theorem \ref{th:0},
$t(QAQ^*)-Qt(A)Q^*=\diag(d)+t(ab^*)$, where $b$ is independent of
$A$. On the other hand equation \eqref{eq:propt1} implies that
$QDQ^*=\diag(d')+t(QDQ^*)$, moreover, since $t(D)=0$, Theorem
\ref{th:0} implies that $t(QDQ^*)=t(a'b^*)$ for some vector
$a'$. Thus we get $QDQ^*=\diag(d')+ t(a'b^*)$. Therefore
$R=\diag(d)+t(ab^*)-\diag(d')-t(a'b^*)=\diag(d+d')+t((a-a')b^*)$,
that is, $R\in\qsh{1}^n$.
\end{proof}

A further analysis enables us to provide an explicit representation of
the $i$th row of the matrix $R$. More precisely we have the following result:
\begin{theorem}\label{th:row}
Under the assumptions of Theorem \ref{th:0}, the $i$th row of the matrix  $R=t(QAQ^*)-Q(D+t(A))Q^*$ has the representation
\begin{equation}\label{eq:gvr}
e_i^TR=[0,\ldots,0,v_i,d_i,w_i^*]G^*_{i-2}G^*_{i-3}\cdots G^*_1
\end{equation}
where, $w_i,d_i\in\mathbb C$, $w_i\in\mathbb C^{n-i}$ and $d_i=r_{i,i}$.
\end{theorem}
\begin{proof}
Let us write $Q=Q_1Q_2$ where $Q_1=G_{1}\cdots G_{i-2}$, $Q_2=G_{i-1}\cdots G_{n-1}$, so that \eqref{eq:gvr} can be rewritten as 
$e_i^TRQ_1=[0,\ldots,0,v_i,d_i,w_i^*]$. This way,
it is enough to show that the $i$th row of $RQ_1$ has the first $i-2$ entries equal to zero.
In view of Lemma \ref{lem:t1}  we have
\[
R':=t(Q_2AQ_2^*)-Q_2(D+t(A))Q_2^*=0_{i-2}\oplus \hat R\in\qsh{1}^n.
\]
Whence
\begin{equation}\label{eq:tmpxx}
Q_2(D+t(A))Q_2^*=t(Q_2AQ_2^*)-0_{i-2}\oplus \hat R.
\end{equation}
Moreover, by definition of $R$ we have
\[
RQ_1=t(Q_1Q_2AQ_2^*Q_1^*)Q_1-Q_1Q_2(D+t(A))Q_2^*.
\]
Setting $B=Q_2AQ_2^*$ and combining the above equation with \eqref{eq:tmpxx} yields
\[
RQ_1=t(Q_1BQ_1^*)Q_1-Q_1t(B)+Q_1(0_{i-2}\oplus \hat R).
\]
Now, since $Q_1(0_{i-2}\oplus\hat R)$ has the first $i-2$ columns
equal to zero, it is sufficient to prove that the $i$th row of
$t(Q_1BQ_1^*)Q_1-Q_1t(B)$ has the first $i-2$ component zero. To this
regard, observe that $Q_1=\tilde Q_1\oplus I_{n-i+1}$, where $\tilde
Q_1\in\mathbb C^{(i-1)\times(i-1)}$, so that partitioning $B$ as
$\left[\begin{smallmatrix}B_{1,1}&B_{1,2}\\ B_{2,1}&B_{2,2}\end{smallmatrix}\right]$,
where $B_{1,1}\in\mathbb C^{(i-1)\times(i-1)}$, by applying again
Lemma \ref{lem:t1} we find that
$t(Q_1BQ_1^*)-Q_1t(B)Q_1^*=W_{i-1}\oplus 0_{n-i+1}$ for some
$W_{i-1}\in\mathbb C^{(i-1)\times(i-1)}$.  This implies that
$t(Q_1BQ_1^*)Q_1-Q_1t(B)$ has the last $(n-i+1)$ rows equal to zero.
This completes the proof.
\end{proof}

Observe that the representation of $R$ that we have found is exactly
the Givens-Vector representation for quasiseparable matrices that is
presented in \cite{vanbarel:book1}. A more general analysis of this
representation is given in Section
\ref{sec:qsr}.

 \section{Reduction to Hessenberg form}\label{sec:rthf}
It is a simple matter to show that the Hessenberg form $H=QAQ^*$
maintains a quasiseparable structure.  In fact, since $D$ is real, we
find that $H=QDQ^* +QUV^*Q^*$ is the sum of a Hermitian matrix
$S=QDQ^*$ and of a matrix $T=QUV^*Q^*$ of rank $k$.  Since $H$ is
Hessenberg, the submatrices contained in the lower triangular part of
$H$ have rank at most 1 so that the submatrices in the lower
triangular part of $S=H-T$ have rank at most $k+1$. On the other hand,
since $S$ is Hermitian, then also the submatrices contained in its
upper triangular part have rank at most $k+1$, thus the submatrices in
the upper triangular part of $H=S+T$ have rank at most $2k+1$ being
the sum of the ranks of the corresponding submatrices of $S$ and $T$
respectively.  This way we find that $H$ is
$(1,2k+1)$-quasiseparable. Actually, we will see later
in Section~\ref{sec:firststep} that the
submatrices contained in the upper triangular part of $H$ have rank
$2k-1$.

The nontrivial problem is to exploit this structure property and to
provide a way to compute the matrix $H$ with a cost lower than
$O(n^3)$ ops needed for general matrices.
In order to arrive at this goal we have to recall the customary
procedure for reducing a matrix $A$ into Hessenberg form which is
based on Givens rotations.
The algorithm can be easily
 described by the pseudo-code reported in Algorithm~\ref{lst:hf}
where the function {\tt givens($v_1,v_2$)} provides the matrix $G(v_1,v_2)$.
      
      \begin{algorithm}
         \caption{Reduction to Hessenberg form by means of Givens rotations}
         \label{lst:hf}
        \begin{algorithmic}[1]
          \For{$j = 1, \dots, n-2$}
            \For{$i = n, \dots, j+2$}
              \State $G \gets \mathtt{givens}(A[i-1,j],A[i,j])$
              \State $A[i-1:i,:] \gets G \cdot A[i-1:i,:]$
              \State $A[:,i-1:i] \gets A[:,i-1:i] \cdot G^*$
            \EndFor
          \EndFor
        \end{algorithmic}
      \end{algorithm}

Denote by $G_{i-1,j}$ the unitary $n\times n$ matrix which performs the
Givens rotation $G={\tt givens}(A( i −1, j ) ,A( i , j ) )$ in the
rows $i-1$ and $i$ at step $j$ of Algorithm \ref{lst:hf} and set
$Q_j=G_{j+1,j}G_{j+2,j}\cdots G_{n-1,j}$. Then  Algorithm  \ref{lst:hf} generates a sequence $A_j$ of matrices such
that
\[\begin{split}
&A_0=A,\quad A_{n-2}=H\\
&A_{j}=Q_{j}A_{j-1}Q_{j}^*, \quad j=1,\ldots,n-2
\end{split}
\]
where the matrix $A_j$ has the form
       \begin{equation}\label{eq:hess}
         A_j=Q_j \dots Q_1 (D + UV^*) Q_1^* \dots Q_j^* = 
         \hat{D} + \hat{U}\hat{V}^* = 
         \begin{bmatrix}
           a^{(j)}_{1,1} & \dots & a^{(j)}_{1,j} & \times & \dots & \times \\
           a^{(j)}_{2,1} & \ddots &  & \vdots & & \vdots \\
           & \ddots & a^{(j)}_{j,j} & \times  & \ldots& \times \\
           & & a^{(j)}_{j+1,j} & \star & \dots & \star \\
           & &           & \vdots &       & \vdots \\
           & &           & \star & \dots & \star
         \end{bmatrix},
       \end{equation}
and the symbols $\times$ and $\star$ denote any numbers.

       For notational simplicity, we denote by $\hat A_j$
       the $(n-j)\times(n-j)$ trailing principal submatrix of $A_j$,
       that is, the submatrix represented by $\star$ in
       \eqref{eq:hess}. Observe that $\hat A_j$ is the part of $A_j$
       that has not yet been reduced to Hessenberg form.  Finally, by
       following the notation of Section \ref{sec:rank-conservation},
       we write $G_i$ or $G_{i,j}$ to denote a unitary matrix
       which applies a Givens transformation in the rows $i$ and
       $i+1$.

The following lemma is useful to get rid of non-generic cases in the process of Hessenberg reduction.

     \begin{lemma} \label{lem:givred}
       Let $v \in \mathbb{C}^n$, $v\ne 0$ and consider Givens
       rotations $G_1, \dots, G_{n-1}$ constructed in such a way that
       $(G_i \dots G_{n-1})v = (w^{(i)*},0,\ldots,0)^*$, where
       $w^{(i)}\in\mathbb C^{i}$, for $i=1,\ldots,n-1$.  If there
       exists $h$ such that $G_h = I$ then one can choose $G_i = I$
       for every $i \geq h$, that is, $(G_1\cdots G_{h-1})v=(w^{(1)*}_1,0,\ldots,0)^*$.
     \end{lemma}
\begin{proof}
Since $G_{i}\cdots G_{n-1}$ is a unitary matrix which acts in the last
$n-i+1$ components of $v$, the 2-norm of $v[i:n]$ coincides with the
2-norm of $(G_i\cdots G_{n-1}v)[i:n]$, that is, $|w^{(i)}_i|$. On the
other hand, if $G_h=I$ then
$(w^{(h)*},0,\ldots,0)=(w^{(h+1)*},0,\ldots,0)$ so that
$w^{(h+1)}_{h+1}=0$. This implies that $\|v[h+1:n]\|=0$, whence
$v_i=0$ for $i=h+1,\ldots,n$. This way, one can choose $G_i=I$ for
$i=h+1,\ldots,n-1$.
\end{proof}
     
Observe that in view of Lemma \ref{lem:unitaryQ} the unitary matrix $Q=G_2\cdots G_{n-1}$ is in upper Hessenberg form. Moreover, the rotations $G_i$ are such that
$QAe_1 = \alpha e_1 + \beta e_2$. In view of Lemma~\ref{lem:givred} we 
can assume that if $G_h = I$ then $G_i = I$ for every $i \geq h$, that is, 
$Q = G_2 \dots G_{h-1}$.  


The quasiseparable structure of the matrices $\hat A_i$ is a consequence of
our main result which is reported in the following

 \begin{theorem} \label{thm:rankstructure}
         Let $U,V,W\in\mathbb C^{n\times k}$,  $S=\diag(d)+t(ab^*)\in\qsh{1}^n$  and
         define
         \[
           A = UV^* + t(UW^*) + S.
         \]
         Let $G_i$, $i=2, \dots,n-1$ be 
         Givens rotations acting on the rows $i$ and $i+1$ such that
         \[
           QA e_1 = a_{1,1} e_1 + \beta e_2,\quad \hbox{where } 
           Q=G_2 \dots G_{n-1} . 
         \]
         Then the matrix $\hat{A}$
         obtained by removing the first row and the first column of
         $QAQ^*$ can be written again as 
         \[
           \hat{A} = \hat{U}\hat{V}^* + t(\hat U\hat W^*) + \hat{S}
         \]
         where $\hat U,\hat W\in\mathbb C^{(n-1)\times k}$, and
         $\hat S=\diag(\hat d+t(\hat a\hat b^*))\in\qsh{1}^{n-1}$ for
         some vectors $\hat d,\hat a,\hat b\in\mathbb
         C^{n-1}$. Moreover, $\hat U$ and $\hat V$ are obtained by
         removing the first row of $QU$ and $QV$, respectively.
\end{theorem}

       \begin{proof} According to Lemma \ref{lem:givred}, we may assume that in the first step of the process 
of reduction in Hessenberg form, the parameters $s_i$ satisfy the
condition $s_i\ne 0$ for $i=2,\ldots,h$, while $s_i=0$, for
$i=h+1,\ldots,n-1$, for some $h\le n-1$.  In view of Lemma
\ref{lem:t1}, without loss of generality we may assume that
$h=n-1$. We have
         \[
           QAQ^* = (QU)(QV)^* + F,\qquad F=Q(t(UW^*)+S)Q^*.
         \]
In view of Corollary~\ref{cor:2}  we have
$F=t(Q(UW^*+ab^*)Q^*)-R$  for $R\in\qsh{1}^n$. Thus
\begin{equation}\label{eq:tmpq}
QAQ^*=(QU)(QV)^*+t(QU(QW)^*)+t(Qa(Qb)^*)-R.
\end{equation}
 Recall from Theorem \ref{th:0} that $Re_1=0$
and that $Q$ has been chosen so that $QAQ^*e_1=\alpha e_1+\beta
e_2$. This fact, together with \eqref{eq:tmpq}, implies that the
vector $u=t(Qa(Qb)^*)e_1$ is such that $u[3:n]$ is in the span of the
columns of $(QU)[3:n,:]$.  In view of \eqref{eq:proptbis} we may write
$t(Qa(Qb)^*)[2:n,2:n]=t(\hat uz^*)$ for $\hat u=u[2:n]$, and for a
suitable $z\in\mathbb C^{n-1}$.  Applying  \eqref{eq:propt}
yields the following representation for the trailing principal
submatrix $\hat A$ of $QAQ^*$ of size $n-1$
\[
\hat A=\hat U\hat V^*+t(\hat U \tilde W^*+\hat u\hat z^*)-\hat R
\]
where $\hat U$, $\hat V$ and $\tilde W$ are obtained by removing the
first row of $U$, $V$ and $W$, respectively, while $\hat
R=R[2:n,2:n]$. Since $\hat u[2:n]$ is in the span of $\hat U[2:n,:]$,
and since the first row of $\hat U$ as well as the first entry of
$\hat u$ do not play any role in the value of $t(\hat U \tilde W^*+\hat
u\hat z^*)$, we may set $\hat u_1$ equal to an appropriate value in such a way
that 
$\hat u$ is in the span of the columns of $\hat U$. This way, the
matrix $\hat U \tilde W^*+\hat u\hat z^*$ has rank at most $k$ and can be
written as $\hat U\hat W^*$ for a suitable $\hat
W\in\mathbb C^{k\times n}$. Thus we have
\[
\hat A=\hat U\hat V^*+t(\hat U \hat W^*)+\hat S
\]
for $\hat S=-\hat R$,
that concludes the proof. 
\end{proof}

Observe that the matrix $A$ defined in \eqref{eq:d+lr} satisfies the
assumptions of the above theorem with $W=0$ and $S=D$ real diagonal. 
This way, Theorem \ref{thm:rankstructure} shows
that the trailing principal submatrix $\hat A_j$ of the sequence generated
by the Hessenberg reduction, maintains the structure $\hat A_j = {U}_j
{V}_j^* + t(U_jW_j^*) + {S}_j$ where $S_j=\diag(d_j)+t(a_jb_j^*)$.

This fact is fundamental to design fast algorithms for
the Hessenberg reduction of a matrix of the form \eqref{eq:d+lr}.  In
order to realize this goal, we need to choose a reliable explicit
representation for $t(U_jW_j^*)$ and $S_j$.  This is the topic of the next
section.

  Note also that the representation $\hat A_j = {U}_j {V}_j^* + t(U_jW_j^*) +
  {S}_j$ implies that $\hat A_j$ is still a quasiseparable matrix of low
  quasiseparability rank. In fact,  $t(U_jW_j^*) + U_jV_j^*$ is (at most) $(k,2k)$-quasiseparable, where
  $k$ is the number of columns of $U_j$. The sum with $S_j$ provides a
  $(k+1,2k+1)$-quasiseparable matrix.  We will prove in
  the next Theorem~\ref{thm:rankreduction} that after the first step it is
  possible to replace $U_j$, $V_j$ and $W_j$ with $n \times (k-1)$
  matrices, thus showing that $\hat A_j$ is a $(k-1,2k-1)$-quasiseparable
  matrix at all the intermediate steps of the Hessenberg reduction.

\subsection{Analysis of the first step} \label{sec:firststep}

       Recall that the matrix $A = A_0$ is of the kind $A =
       D+UV^*$. This is a particular case of the form $A = UV^* +
       t(UW^*)+S$ above where $W = 0$ and $S$ is diagonal. Because of
       this additional structure, the first Hessenberg reduction step
       is somewhat special and we can get a sharper bound for the
       quasiseparability rank of $\hat{A_1}$. This is shown in the next

       \begin{theorem} \label{thm:rankreduction}
        Assume that the matrix $A$ of \eqref{eq:d+lr} does not have
        the first column already in Hessenberg form, i.e.,
        $A[3:n,1]\ne 0$.  Then the rank of the matrix obtained by
        removing the first two rows of $Q_1U$ is less than $k-1$.
        Moreover, the $(n-1)\times(n-1)$ trailing principal submatrix
        $\hat A_1$ of $A_1=Q_1AQ_1^*$ has lower quasiseparable rank
        $k$.
       \end{theorem}
       
       \begin{proof}
         Observe that $\hat A_1$ can be written as the sum of a
         Hermitian matrix and of a matrix of rank at most $k$, namely,
         $ \hat{A}_1 = \hat{Q}_1 D \hat{Q}_1^* + \hat{U}_1
         \hat{V}_1^*, $ where $\hat U_1$, $\hat V_1$ and $\hat Q_1$
         are the matrices obtained by removing the first row of
         $Q_1U$, $Q_1V$, and $Q_1$, respectively. Define $x=V^*e_1$ so
         that $x\ne 0$ and $Ae_1=d_1e_1+Ux$. Since $Q_1 A e_1 =
         \alpha_1 e_1 + \beta_1 e_2$, we find that $Q_1U x
         =Q_1Ae_1-d_1Q_1e_1=(\alpha_1-d_1)e_1+\beta_1e_2$, which
         implies $\hat U_1x=\beta_1 e_1$.  Thus, $\hat U_1[2:n-1,:]$
         has rank at most $k-1$. This matrix coincides with the matrix
         obtained by removing the first two rows of $Q_1U$ so that the
         first part of the theorem is proven.  Since $\hat
         U_1[2:n-1,:]$ has rank at most $k-1$, then also the matrix
         $\hat U_1[2:n-1,:]\hat V_1^*$ has rank at most $k-1$.  We can
         conclude that every submatrix in the strictly lower
         triangular part of $\hat A_1$, given by the sum of a
         submatrix of $\hat U_1[2:n-1,:]\hat V_1^*$ which has rank at
         most $k-1$, and a submatrix in the lower triangular part of
         $Q_1 D Q_1^*$ which has rank at most $1$ in view of
         Corollary \ref{cor:2}, can have rank at most $k$. This
         completes the proof.
       \end{proof}

       It is important to note that we are not tracking the structure
       of the whole matrix but only the structure of the trailing
       part that we still need to reduce. This is not a drawback
       since the trailing part is the only information needed to continue the
       algorithm. Moreover, the entire Hessenberg matrix can be recovered at
       the end of the process by just knowing the diagonal and
       subdiagonal elements that are computed at each step together with the
       matrices $U_{n-2}$ and $V_{n-2}$.

\section{Representing quasiseparable matrices}\label{sec:qsr}
    
    Finding good and efficient representations of quasiseparable
    matrices is a problem that has been studied in recent years. 
    Some representations have been introduced and analyzed for the 
    rank $1$ case. Some of them have been extended to (or were originally
    designed for) the higher rank case. 
    
    In this section we provide a representation of quasiseparable
    matrices which, combined with the results of the previous section,
    enables us to design an algorithm for the Hessenberg reduction of
    the matrix $A$ with cost $O(n^2k)$.
    
    \subsection{Givens Vector representations}
    
    A useful family of representations for quasiseparable matrices
    is described in \cite{vanbarel:book1} (for the $1$-quasiseparable case) 
    and is extended to a more general version in \cite{van2005diagonal}. 
We use this kind of representation and adjust it to our framework.
For the sake of clarity, we provide the details of the new notation in order to 
    make this section self-contained. 
    
    \begin{defi}\label{def:gseq}  A tuple $\mathcal G = \left( G_i \right)_{i \in I}$
      on some ordered index set $(I,\leq)$
      is said a \emph{sequence} of Givens rotations. We also define $\prod_{i \in I}G_i$ the product
      in ascending order while $\prod_{i \in \mathrm{rev(I)}}G_i$ denotes 
      the product in descending order. The following
      operations on $\mathcal G$ are introduced:
      \begin{itemize}
        \item 
          $\mathcal G v := \prod_{i \in \mathrm{rev(I)}} G_i v$, for $v\in\mathbb C^n$;
        \item 
          $\mathcal G^* v := \prod_{i \in I} G_i^* v$ for $v\in\mathbb C^n$;
        \item for $J \subseteq I$, with the induced
          order, we call $\mathcal G[J] := \left( G_j \right)_{j \in J}$
          the \emph{slice} of $\mathcal G$ on the indices $J$; 
        \item for Givens sequences $\mathcal G = \left(G_i \right)_{i \in I}$, 
        $\mathcal G' = \left( G'_i \right)_{j \in J}$ , we define
          the product $\mathcal G \mathcal G'$ to be the sequence
          \[
            \mathcal G \mathcal G' := \left( E_i \right)_{i \in I \sqcup J}, 
            \qquad E_i = \begin{cases} G_i \text{ if } i \in I \\
            G'_i \text{ if } i \in J \end{cases}
          \]
          where $\sqcup$ is the disjoint union operator
          and where the order on $I \sqcup J$ is induced by the ones
          on $I$ and $J$ and by the agreement that $G_i < G'_j$ for every
          $i \in I, j \in J$. 
      \end{itemize}      
      The above definitions on the product between a sequence and a vector
      trivially extend to products between sequences and matrices. For instance,
$\mathcal G A:= \prod_{i \in \mathrm{rev(I)}} G_i A$.
    \end{defi}
    
    We are interested in the cases where the index sets $I$ are
    special, in particular we consider the case of univariate sets
    $I\subset \mathbb N$, and the case of bivariate sets
    $I\subset\mathbb N^2$. Let us give first the more simple definition
    of $1$-sequences, which covers the case of univariate sets, and
    then extend it to $k$-sequences for $k>1$, that is, the case of
    bivariate sets.
 
    \begin{defi} \label{def:sequence}
      We say that $\mathcal G$ is a $1$-\emph{sequence} of Givens rotations if 
      $\mathcal G = \left( G_2, \dots, G_{n-1} \right)$. 
    \end{defi}
    
      Notice
      that in this context, the operations already introduced in Definition \ref{def:gseq} 
 specialize in the following way:
      \begin{itemize}
        \item $\mathcal G v := G_{n-1} \dots G_2 v$, for $v \in \mathbb{C}^{n}$;
        \item  $\mathcal G^* v := G_{2}^* \dots G_{n-1}^* v$, for $v \in \mathbb{C}^{n}$;
                 
        \item $\mathcal G[i:j] := 
          ( G_i, \dots, G_j )$, for $2 \leq i < j \leq n-1$, is a \emph{slice} of $\mathcal G$ from $i$ to $j$. 
      \end{itemize}
     Here, and hereafter, we use the notation $i:j$ to mean the tuple
     $( i, i+1, \dots, j )$ for $i<j$.  Sometimes we use the
     expressions $\mathcal G[i:]$ or $\mathcal G[:j]$ for
     $\mathcal G[i:n-1]$ and $\mathcal G[2:j]$, respectively. That is,
     leaving the empty field before and after the symbol ``$:$'' is a
     shortcut for ``starting from the first rotation'' and for ``until
     the last rotation'', respectively.

Below, we recall a useful pictorial representation, introduced in
\cite{vanbarel:book1} and \cite{van2005diagonal}, which effectively
describes the action of the sequence of Givens rotations. We report
the case $n=6$, where every $\gr$ describes a Givens rotation $G_i$ applied
to the pair $(i,i+1)$ of consecutive rows.
    \[
      \mathcal G = ( G_2, \dots, G_n ) = \raisebox{-7pt}{\descendinggrfull} , \qquad
      \mathcal Gv = \raisebox{-7pt}{\descendinggrfull} \left[ \begin{array}{l}
        v_1 \\ v_2 \\ v_3 \\ v_4 \\ v_5 \\ v_6
      \end{array} \right].
    \]

    The definition of 1-sequences is generalized to the case where $I\subset\mathbb N^2$ is a 
    bivariate set of indices, and  where we consider Givens rotations
   $G_{i,j}$ acting on the pair $(i,i+1)$ of consecutive rows for any $j$.
   In this case, the ordering on $\mathbb N^2$ which induces orderings in any subset $I$ of $\mathbb N^2$,  is defined by
    \[
      (i_1,j_1) \leq_G (i_2,j_2) \iff j_1 > j_2 \text{ or } \left( j_1 = j_2 
       \text{ and } i_1 \leq i_2 \right).
    \]
    
    \begin{defi}
      We say that $\mathcal G=(G_{i,j})_{(i,j)\in I}$ is a
      $k$-sequence of Givens rotations if $I = \{ (i,j)\in\mathbb N^2
      \; | \; i = 2, \dots, n-1, j = 1, \dots, \min(i-1,k) \}$ with
      the order induced by $\leq_G$.  With a slight abuse
      of notation we define the sequence $ \mathcal
      G[i_1:i_2] :=(G_{i,j})_{(i,j)\in I'}$, $I'=\{ (i,j)\in\mathbb
      N^2 \; | \; i = i_1, \dots, \min(i_2+k,n-1), j = \max(1,i-i_2+1), \dots, \min(k,i-i_1+1), \quad
      \: 2 \leq i_1 < i_2 \leq n-1\}$ to be a \emph{slice} of
      $\mathcal G$ from $i_1$ to $i_2$, where the ordering in $I'$ is
      induced by the ordering $\leq_G$ valid on the parent set.
    \end{defi}
    
    A pictorial representation  similar to the one given above 
    can be used also in this case. For example, for $k = 2$ and $n=6$ we have
\[\mathcal G=(G_{3,2}, G_{4,2}, G_{5,2}, G_{2,1}, G_{3,1}, G_{4,1}, G_{5,1})\] that is represented by
    \begin{equation} \label{eq:ksequencepicture}
      \mathcal G = \descendinggrfull \descendinggr. 
    \end{equation}
    
    \medskip
    
    Note that for every $i_1 \leq i_2 < i_3$, the slices of $\mathcal G$
    can be factored in the following form:

    \[
      \mathcal G[i_1:i_3] = \mathcal G[i_2+1:i_3] \mathcal G[i_1:i_2], \qquad
      \mathcal G^*[i_1:i_3] = \mathcal G^*[i_1:i_2] \mathcal G^*[i_2+1:i_3],
    \]
where, for notational simplicity, we set
  $\mathcal G^*[i_1:i_2]=\left(\mathcal G[i_1:i_2] \right)^*$.
This property is called {\em slicing of rotations}.
    
    
    Note that the order $\leq_G$ is one of the orders such that 
    $\mathcal G v$ coincides with the multiplication of the vector
    $v$ by the Givens rotations in $\mathcal G$ with the order induced
    by the pictorial representation \eqref{eq:ksequencepicture}. 
    
    It is worth highlighting that the operation of slicing
    a $k$-sequence is equivalent to removing the heads and tails
    from the sequences itself. 
    For example the slice of $\mathcal G$
    defined by $\mathcal G[3:n-2]$ is obtained by taking only the bold
    rotations in the following picture, where $n = 7$, which correspond to
    $G_{4,2},G_{5,2},G_{3,1},G_{4,1}$.
    \[
      \mathcal G = \begin{array}{lllll}
                      &&&&\gr\\
                      &&&\redgr \\
                      && \redgr \\
                      &\gr \\
                      \gr \\
                    \end{array}
                    \begin{array}{llll}
                                    \\[1ex]
                                    &&&\gr\\
                                    &&\redgr \\
                                    &\redgr \\
                                    \gr \\
                                  \end{array}.
    \]

   With the basic tools introduced so far we can define the concept of Givens
    Vector representation.
    
    \begin{defi}\label{def:gvr}
      A {\em Givens Vector (GV) representation} of rank $k$ for a Hermitian
      quasiseparable matrix $A$ is a triple $(\mathcal G, W, D)$
      where $\mathcal G$ is a $k$-sequence of Givens rotations, $W \in \mathbb{C}^{k \times (n-1)}$ and $D$ is a diagonal matrix such that
      \begin{itemize}
        \item $D$ is the diagonal of $A$;
        \item for every $i = 1, \dots, n-1$ the subdiagonal elements
         of the $i$-th column of $A$ are equal to the last $n-i$ elements
         of $\mathcal G[i+1:] \underline{w}_i$,   where we define
         \[
           \underline{w}_i := \begin{bmatrix}
             0_{i} \\
             We_i \\
             0_{n-k-i} \\
           \end{bmatrix} \: \text{if } i < n-k, \qquad
           \underline{w}_i := \begin{bmatrix}
            0_{i} \\
            (We_i)[1:n-i] \\
           \end{bmatrix} \: \text{otherwise}
         \]
         where $0_j$ is the $0$ vector of length $j$ if 
         $j > 0$, and is the empty vector otherwise. 
 That is, $\tril(A,-1)e_i=\mathcal G[i+1:] \underline{w}_i$.
      \end{itemize}
      If the triple $(\mathcal G, W, D)$ is a GV
      representation of the matrix $A$ we write that $A = \gv{\mathcal G,
        W, D}$.
    \end{defi}

  We refer to \cite{van2005diagonal} for a detailed analysis of the properties
  of this representation. We recall here only the following facts: 
  
  \begin{itemize}
    \item If $A$ is $k$-quasiseparable then there exists a $k$-sequence
    $\mathcal G$, a matrix $W \in \mathbb{C}^{k \times (n-1)}$ and a diagonal
    matrix $D$ such that $A = \gv{\mathcal G, W, D}$. 
    \item If $A = \gv{\mathcal G, W, D}$ for some $k$-sequence $\mathcal G$, 
    $W \in \mathbb{C}^{k \times (n-1)}$ and $D$ diagonal, then $A$ is 
    $j$-quasiseparable with $j \leq k$, i.e., the existence of the representation guarantees that $A$ is at most $k$-quasiseparable.  
  \end{itemize}
  
  We introduce now an important operation on Givens rotations, called
  \emph{turnover}. The following Lemma can be thought as a partial
  answer to the question whether two Givens rotations commute. It is
  clear that if we have $G_i$ and $G_j$ such that $\lvert i - j \rvert
  > 1$ then $G_iG_j - G_jG_i = 0$. This is also true if the two
  rotations act on the same rows, but it does not hold when they are
  acting on consecutive rows. In the latter case, the turnover gives a
  way to swap the order of the rotations.
  
  \begin{lemma} \label{lem:turnover}
    Let $\mathcal G$ be a sequence of Givens rotations and 
    $F_i$ a Givens rotation acting on the rows $i$ and $i+1$.
    Then there exists another sequence $\hat{\mathcal G}$ and a Givens rotation 
    $\hat{F}_{i-1}$ acting on the rows $i-1$ and $i$ such that
    \[
      \mathcal G F_i = \hat{F}_{i-1} \hat{\mathcal G}. 
    \]
    Moreover, $\hat{\mathcal G}$ differs from $\mathcal G$ only
    in the rotations acting on the rows with indices $(i-1,i)$ and $(i,i+1)$. 
  \end{lemma}
  
  See \cite{vanbarel:book1} for a proof of this fact. The pictorial
  representation of the Givens rotations can be helpful to understand
  how the turnover works.
  \[
    \begin{array}{llll}
      &&&\gr \\
      &&\gr \\
      & \gr && \gr \\
      \gr \\
      \end{array} \qquad = \qquad
      \begin{array}{llll}
        &&&\gr \\
        \hat\gr &&\hat\gr \\
        & \hat\gr && \\
        \gr \\
        \end{array}
  \]
  The above lemma
  can be easily extended to $k$-sequence of rotations.

  \begin{cor}
    Let $\mathcal G$ be a $k$-sequence of Givens rotations and $F_i$ a
    Givens rotation acting on the rows $i$ and $i+1$.  Then there
    exists another $k$-sequence $\hat{\mathcal G}$ and a Givens
    rotation $\hat{F}_{i-k}$ acting on the rows $i-k$ and $i-k+1$ such
    that
    \[
      \mathcal G F_i = \hat{F}_{i-k} \hat{\mathcal G},
    \]
where $\hat F_{i-k}=I$ if $i-k\le 0$.
    Moreover, $\hat{\mathcal G}$ differs from $\mathcal G$ only
    in the rotations of indices $(i-j+1,j)$ and $(i-j,j)$ for $j = 1, \dots, k$. 
  \end{cor}
  
  Again, a pictorial representation of this fact can be useful to
  figure out the interplay of the rotations. Below, we report the case
  where $i>k+1$.
  
  \[
    \begin{array}{llllll}
      &&&& \gr & \\
      &&&\gr && \gr \\
      && \gr && \gr \\
      & \gr && \gr && \gr \\
      \gr && \gr \\
      \end{array} \qquad = \qquad
          \begin{array}{llllll}
            &&&& \gr & \\
            &\hat\gr &&\hat\gr && \gr \\
            && \hat\gr && \hat\gr \\
            & \gr && \hat\gr && \\
            \gr && \gr \\
            \end{array} 
  \]
  
  \medskip
  
  The above operations are very cheap. The cost of the computation
  of a turnover is $O(1)$
  in case of $1$-sequences and $O(k)$ in case of $k$-sequences. 
  
  We answer now to the following question: Given a $k$-sequence $\mathcal G$
  and a matrix $A$ of quasiseparable rank $k$, there exist appropriate matrices
  $W \in \mathbb{C}^{k \times (n-1)}$ and $D$ diagonal such that 
  $A = \gv{\mathcal G, W, D}$?
  
  \begin{lemma} \label{lem:qscheck}
    Let $A$ be a Hermitian matrix and $\mathcal G$ a $k$-sequence
    of Givens rotations. Then $B=\mathcal G^*A$ is lower banded
    with a bandwidth of $k$, i.e., $b_{i,j}=0$ for $i-j>k$, if and only if 
    the matrix $A$ admits a representation
    of the form $\gv{\mathcal G, W, D}$ for some $W \in \mathbb{C}^{k \times n}$ and $D$ real diagonal. 
  \end{lemma}
  
  \begin{proof}
    We first suppose that $A = \gv{\mathcal G, W, D}$. 
    Recall that, by definition of GV representation, $\tril(A,-1) e_i = \mathcal G[i+1:] \underline w_i$ for $i=1,\ldots,n-1$. This implies that 
    \[
      \mathcal G^* \tril(A,-1) e_i = \mathcal G^* \mathcal G[i+1:] \underline w_i 
      = \mathcal G^*[:i] \mathcal G^*[i+1:] \mathcal G[i+1:] \underline w_i =  \mathcal G^*[:i] \underline w_i. 
    \]
    We also have
    \[
      \mathcal G^* \triu(A) e_i = 
      \mathcal G^*[:i] \mathcal G^*[i+1:] \triu(A) e_i = 
      \mathcal G^*[:i] \triu(A) e_i,
    \]
    since $G^*[i+1:]$ is acting on rows that are null. 
    So by decomposing $A = \tril(A,-1) + \triu(A)$ we have
    \[
      \mathcal G^* A e_i = \mathcal G^* \triu(A) e_i + 
        \mathcal G^* \tril(A,-1) e_i = \mathcal G^*[:i] (\triu(A) e_i
          + \underline w_i ).
    \]
    Now observe that the rotations inside $\mathcal G^*[:i]$ only
    act on the first $i+k$ rows. This implies that, since both $\underline w_i$ and $\triu(A) e_i$
    have all the components with index strictly
    bigger than $i+k$ equal to zero, 
    the same must hold for $\mathcal G^*[:i] (\underline w_i + \triu(A)e_i)$, and 
    this completes the proof. The converse is also true. In fact, if $\mathcal G^* A$ is lower banded with bandwidth $k$
    we can build $W$ by setting
    $We_i = \left( \mathcal G^*[i+1:] A e_i \right)[i+1:i+k]$ and $D$
    equal to the diagonal of $A$. Then the equation $A = \gv{\mathcal G, W, D}$
    can be verified by direct inspection. 
  \end{proof}

  To simplify the notation when talking about ranks in the lower part
  of quasiseparable matrices, we say that a $k$-sequence  $\mathcal G$ 
  \emph{spans} $U\in\mathbb C^{n\times k}$ if
  there exists $Z\in\mathbb C^{k\times k}$ such that $\mathcal G^* U = \left[\begin{smallmatrix}
   Z \\
   0 \\
  \end{smallmatrix}\right]$. This definition is motivated by the
  following
  
  \begin{lemma} \label{lem:qrcheck}
    If $\mathcal G$ spans $U \in \mathbb{C}^{n \times k}$ then,  
    for every $V \in \mathbb{C}^{n \times k}$,
    $W \in \mathbb{C}^{k \times (n-1)}$ and $D$ diagonal,
    the matrix $A_1 = UV^* + \gv{\mathcal G, W, D}$
    is lower $k$-quasiseparable and $A_2 = t(UV^*) + \gv{\mathcal G, W, D}$
    is $k$-quasiseparable. In particular, both $\mathcal G^* A_1$ and
    $\mathcal G^* A_2$ are lower banded with bandwidth $k$. 
  \end{lemma}
  
  \begin{proof}
    For the first part of the Lemma it suffices to observe 
    that $\mathcal G^* A_1$ is lower banded
    with bandwidth $k$. This follows directly by noting that 
     $A=\gv{\mathcal G, W, D}+UV^*$. Since $\mathcal G^* UV^* = \left[ \begin{smallmatrix}
      Z \\ 0
    \end{smallmatrix} \right] V^*$ and $\mathcal G^* \gv{\mathcal G, W, D}$
    is lower banded by Lemma~\ref{lem:qscheck}, we conclude that
    also $\mathcal G^* A_1$ is lower banded with bandwidth $k$. Since
    the strictly lower part of $A_2$
    coincides with the one of $A_1$ we find that
    also $A_2$ is lower $k$-quasiseparable. Given that $A_2$ is Hermitian, 
    we conclude that $A_2$ is also upper $k$-quasiseparable. To see that
    also $\mathcal G^* A_2$ is lower banded we can write
    \[
      \mathcal G^* A_2 = \mathcal G^* (A_1 - \triu(UV^*) + \triu(VU^*, 1)) = 
        \mathcal G^* A_1 + \mathcal G^* R,
    \]
    where $R$ is upper triangular. 
    Since $\mathcal G^*$, represented as a matrix, is the product of $k$ upper Hessenberg matrices, it is lower banded with bandwidth $k$. This implies that also $\mathcal G^* R$ is lower banded with bandwidth $k$ and so the same must hold also for $\mathcal G^* A_2$. 
  \end{proof}
  
  \begin{rem} \label{rem:qrlink}
    The above Lemma shows how the Givens rotations in a GV representation of 
    a matrix in $\qsh{k}$ are sufficient to determine the column span
    of the submatrices contained in the lower triangular part. These
    matrices give the same information obtained by knowing the matrix $U$
    in the $D + t(UV^*)$ representation. 
  \end{rem}
      
  We need to find efficient algorithms to perform operations on this
  class of matrices. More precisely, in order to implement in terms of
  algorithms the constructive proofs given in
  Section~\ref{sec:rank-conservation}, we need to explain how to
  efficiently perform the following tasks assuming we are given GV
  representations of $M=D+t(UV^*)=\gv{\mathcal G,W,D}\in\qsh{k}$ and of
  $S=D_S+t(uv^*)\in\qsh{1}$, where $U,V\in\mathbb C^{n\times k}$,
  $\mathcal G$ spans $U$,
  $u,v\in\mathbb C^n$ and $u=Ux$ for some vector $x\in\mathbb C^k$:
  
  \begin{enumerate}
    \item Compute a GV representation of rank $k$ of $M + S$.
    \item Given a unitary upper Hessenberg matrix $P$, compute a GV
      representation of rank $k$ of $t(PU(PV)^*)$, and a GV
      representation of rank 1 of $R = PMP^* -t(PU(PV)^*)$.
  \end{enumerate}
  
   We start by analyzing the problem of computing $M+S$.   Since  $S=D_S+t(Uxv^*)$, in view
   of Lemma~\ref{lem:qrcheck}, we find that $\mathcal G^*S$ is lower
   banded with bandwidth $k$ and so by applying
   Lemma~\ref{lem:qscheck} there exists $W_S \in \mathbb{C}^{n \times
     k}$ and $\hat{D_S}$ real diagonal such that $S=\gv{\mathcal{G},
     W_S, \hat{D_S}}$ is a GV representation of $S$. Given an
   algorithm for the computation of $W_S$ it is possible to represent
   $M+S$ as $\gv{\mathcal G, W_S+W, D + D_S}$.
  
  We need to investigate how to actually compute the matrix
  $W_S$ assuming we are given  a GV representation
  $S = \gv{\mathcal F, z, D_S}$ of $S$.
   Recall that the $i$-th
  column of $W_S$ can be extracted from the components of the 
  vector $\mbox{\ensuremath{\mathcal G^*[i+1:]}} M e_i$, as explained in Lemma~\ref{lem:qscheck}. We can compute the whole matrix $W_S$
  at cost $O(nk)$ by following this procedure: 
  
  \begin{itemize}
    \item Compute the last column of $W_S$ by using Lemma \ref{lem:qscheck}.
      This is almost cost-free since no rotations are involved, and the only
      significant element of $W_Se_{n-1}$ is equal to $z_{n-1}$. 
    \item Compute $W_Se_{i}$ starting from $W_Se_{i+1}$; 
      this vector can be computed by using some
      elements in $\mathcal G^*[i+1:] M e_i$. In fact, since we are in
      the $1$-quasiseparable case then $\underline z_i=z_i e_{i+1}$. So we have
      \begin{align*}
        \mathcal G^*[i+1:] M e_{i} &= 
          \mathcal G^*[i+1:] \mathcal F[i+1:] \underline z_{i} = 
          \mathcal G^*[i+1] \mathcal G^*[i+2:] \mathcal F[i+2:] \mathcal F[i+1] \underline z_{i}\\
          &= z_i \mathcal G^*[i+1] \mathcal G^*[i+2:] \mathcal F[i+2:] \left(
           \alpha e_{i} + \beta e_{i+1} \right) = \\
          &= z_i \mathcal G^*[i+1] \left( \frac{\beta}{z_{i+1}} \mathcal G^*[i+2:]  M e_{i+1} + \alpha e_i \right). 
      \end{align*}
      Since $\mathcal G^*[i+2:]  Me_{i+1}$ has already been computed,
      we obtain $W_Se_{i}$ from $W_Se_{i+1}$ at the cost of $O(k)$
      rotations. Some care needs to be taken because in the above formula
      it may happen that $z_{i+1}$ is zero. This issue can be solved
      by replacing $M e_i$ with $\mathcal F[i+1:] e_i$ and by multiplying
      the computed $W_Se_i$ by $z_i$ after the computation. This way, we can 
      also avoid the possible cancellation in cases where the $z_i$ are
      unbalanced.
  \end{itemize}
  
  The above procedure provides a fast method for computing  
  a $k$-quasiseparable representation of $S$.
  
  We investigate now the problem of computing the residual matrix
  given by Theorem~\ref{th:0} and Corollary \ref{cor:2} using the GV
  representation of $M$ and $S$.  We rephrase the proof of Theorem~\ref{th:0} 
in terms of GV representations.

  Analyzing the proof of Theorem~\ref{th:0} and its corollaries, we
  can observe that the algorithm can be easily constructed if we are
  able to compute the residual $R_i = t(F_iUVF_i^*) - F_it(UV^*)
  F_i^*$ for a Givens rotation $F_i$ acting on the rows $(i,i+1)$. 
 In the following, we suppose that $i < n - k - 1$ so
  that we do not need to care about ``border conditions''. However, all the
  concepts reported are easily extendable to those cases by
  just adding some care in the process.
  
 Assume we are given $D,U,V,W$ and $\mathcal G$ such that
  $M = D + t(UV^*) = \gv{\mathcal G, W, D}$. Observe that we can
  compute an updated $\hat{\mathcal G}$ such that $\hat{\mathcal G}$ spans 
  $F_i U$. In fact, we know that $\mathcal G^* F_i^* F_i U$ is of the form $\left[\begin{smallmatrix}
    \times \\ 0
  \end{smallmatrix}\right]$ where $\times$ is an appropriate
  $k \times k$ block. The rotation $F_i^*$ can be passed through the
  rotations inside $\mathcal G^*$ (by properly updating them using the
  turnover operation)
  obtaining $\hat{F}_{i+k}^* \hat{\mathcal G}^* = \mathcal G^*
  F_i^*$. Then, by $\hat{F}_{i+k}^* \hat{\mathcal
    G}^* F_i U = \left[\begin{smallmatrix} \times \\ 0
    \end{smallmatrix}\right]$, we can conclude that also $\hat{\mathcal G}^* F_iU = \hat F_{i+k} \left[\begin{smallmatrix}
          \times \\ 0
        \end{smallmatrix}\right] = \left[\begin{smallmatrix}
          \times \\ 0
        \end{smallmatrix}\right]$ since $\hat{F}_{i+k}$ is operating
        on the null rows.
        
  Moreover, we can check that $\gv{\hat{\mathcal G}, W, D}$ correctly
  represents the lower part of $t(F_i(D + UV^*)F_i^*)$ on every column but
  the one with indices $i, i+1$. In fact, the diagonal part of $M$ is 
  left unchanged on the indices different from $i,i+1$. For the rest of the
  matrix we can distinguish two cases and we do not need to care about
  $D$:
  \begin{itemize}
    \item If $j > i+1$, both the left multiplication
      by $F_i$ and the right multiplication by $F_i^*$ leave unchanged
      the relevant part of $U$ and $V$ needed for the computation
      of the portion of the $j$-th column contained in the lower part
      of the matrix. Moreover, since in this case
      $\hat{\mathcal G}[j+1:] = \mathcal G[j+1:]$ we conclude that the proposed representation 
      for these columns is valid. 
    \item Also when $j < i$ the right multiplication by 
      $F_i^*$ does not change the $j$-th column at all. However, the 
      left multiplication by $F_i$ does  change the $j$-th column and we can 
      verify that 
      $\tril(t(F_iUV^*F_i^*),-1) e_j = \tril (F_i U V,-1) e_j$. Recall
      that, by definition of GV representation, we have
      \[
        \tril(t(UV^*),-1) e_j = \tril(UV^*,-1) e_j = \tril(M,-1) e_j = 
        \mathcal G[j+i:] \underline{w_i}. 
      \]
      Since $j < i$ we have  $F_i \tril(UV^*,-1) e_j = 
      \tril(F_iUV^*,-1) e_j$ so that we can write
      \[
        \tril(F_iUV^*,-1) e_j = F_i \mathcal G[j+1:] \underline w_i = 
        \hat{\mathcal G}[j+1:] \hat{F}_{i+k} \underline w_j = 
        \hat{\mathcal G}[j+1:] \underline w_j
      \]
      where the last two equalities follow from the definition of 
      $\hat{\mathcal G}$ and from the fact that the components of $\underline w_j$
      with index bigger than $j+k$ are zero. Thus we have
      that $\gv{\hat{\mathcal G}, W, D}$ provides a good representation
      of the lower part of the $j$-th column of $t(F_iUV^*F_i^*)$, as 
      requested. 
  \end{itemize}
  
  A pictorial representation of these two facts can help to get a better
  understanding of what is going on (here we are fixing $k = 2$). The rotation
  $F_i$ on the left is highlighted using the bold font. Equation~\eqref{eq:gvupdate1} represents the first case, where the rotation
  $F_i$ does not intersect the indices of the rotations in $\mathcal G[j+1:]$, and Equation~\eqref{eq:gvupdate2} the latter case, where an
  update of the rotations is necessary. 
  
  \begin{equation} \label{eq:gvupdate1}
    \begin{array}{lllll}
      \\ &&&& \redgr \\ \\ 
      &&& \gr \\
      && \gr && \gr \\
      & \gr && \gr \\
      \gr && \gr \\
    \end{array} \left[ \begin{array}{l}
      0 \\ \vdots \\ 0 \\ \star \\ \star \\ 0 \\ \vdots \\ 0 \\
    \end{array}\right] = 
        \begin{array}{lllll}
          \\ \\ \\ 
          &&& \gr \\
          && \gr && \gr \\
          & \gr && \gr \\
          \gr && \gr \\
        \end{array} \left[ \begin{array}{l}
          0 \\ \vdots \\ 0 \\ \star \\ \star \\ 0 \\ \vdots \\ 0 \\
        \end{array}\right]
  \end{equation}
  
  \begin{equation} \label{eq:gvupdate2}
      \begin{array}{lllll}
        \\ \\ \\ 
        &&& \gr \\
        \redgr && \gr && \gr \\
        & \gr && \gr \\
        \gr && \gr \\
      \end{array} \left[ \begin{array}{l}
        0 \\ \vdots \\ 0 \\ \star \\ \star \\ 0 \\ \vdots \\ 0 \\
      \end{array}\right] = 
          \begin{array}{lllll}
            \\ \\ \\ 
            &&& \gr \\
            && \hat\gr && \gr \\
            & \hat\gr && \hat\gr \\
            \gr && \hat\gr && \hat{\redgr} \\
          \end{array} \left[ \begin{array}{l}
            0 \\ \vdots \\ 0 \\ \star \\ \star \\ 0 \\ \vdots \\ 0 \\
          \end{array}\right] = 
                    \begin{array}{lllll}
                      \\ \\ \\ 
                      &&& \gr \\
                      && \hat\gr && \gr \\
                      & \hat\gr && \hat\gr \\
                      \gr && \hat\gr  \\
                    \end{array} \left[ \begin{array}{l}
                      0 \\ \vdots \\ 0 \\ \star \\ \star \\ 0 \\ \vdots \\ 0 \\
                    \end{array}\right]
  \end{equation}

  This means that we need to track only what happens on columns
  $(i,i+1)$. We show how to update $D$ and $W$ in the $i$ an $i+1$
  components in order to account for what happens in these indices. 
  Note that these columns of $M$ can be described in
  the following way (we report the case $k = 3$ for simplicity):     
  
  \[
    M \begin{bmatrix} 
      \\
      e_i & e_{i+1} \\ 
      \\ \end{bmatrix} = \mathcal G[i+2:] \left( 
      \ascendinggr
       \begin{bmatrix}
        d_i & 0  \\
        w_{1,i} & 0 \\
        w_{2,i} & 0 \\
        w_{3,i} & 0 \\
        0 & 0
      \end{bmatrix} +  \begin{bmatrix}
      0 & \times  \\
      0 & d_{i+1} \\
      0 & w_{1,i+1} \\
      0 & w_{2,i+1} \\
      0 & w_{3,i+1} \\
    \end{bmatrix}  \right).
  \]
  
  Left and right multiplying by $F_i$ and $F_i^*$
  (reported with the bold font), respectively, leads to the following
  structure: 
  
  \[
    F_i M F_i^* \begin{bmatrix} 
      \\
      e_i & e_{i+1} \\ 
      \\ \end{bmatrix} = \mathcal G[i+2:] \left( 
\raisebox{6pt}{\ensuremath{\begin{array}{llll}
            \\
            \redgr \\
            &\gr \\
            &&\gr \\
            &&&\gr \\
          \end{array}}}
       \begin{bmatrix}
        d_i & 0  \\
        w_{1,i} & 0 \\
        w_{2,i} & 0 \\
        w_{3,i} & 0 \\
        0 & 0
      \end{bmatrix} +  \begin{array}{l} 
        \redgr \\
        \\
        \\
        \\
      \end{array}
      \begin{bmatrix}
      0 & \times  \\
      0 & d_{i+1} \\
      0 & w_{1,i+1} \\
      0 & w_{2,i+1} \\
      0 & w_{3,i+1} \\
    \end{bmatrix}  \right) \begin{array}{l}
     \redgr \\ \\ \\ \\ \end{array}.
  \]  
  
  We can explicitly compute the value inside the brackets and
  then observe that, since \mbox{$\hat{\mathcal G}[i+2:]$} $= 
  \mathcal G[i+2:]$, we have a representation of the columns
  of $F_iMF_i^*$. Now we want to find a Hermitian matrix
  $R$ of the form $R = \alpha  e_{i+1}e_i^t + \conj \alpha e_ie_{i+1}^t$,  $\hat{w}_{j,i}$, $\hat{w}_{j,i+1}$ for $j = 1, \dots, k$ and $\hat{d_i}$, 
  $\hat{d}_{i+1}$ such that,
  writing with $\hat{\gr}$ the rotations taken from $\hat{\mathcal G}$, we have

{\small  \begin{equation} \label{eq:conjugation}
     \underbrace{\left( 
\raisebox{6pt}{\ensuremath{\begin{array}{llll}
            \\
            \redgr \\
            &\gr \\
            &&\gr \\
            &&&\gr \\
          \end{array}}}
           \begin{bmatrix}
            d_i & 0  \\
            w_{1,i} & 0 \\
            w_{2,i} & 0 \\
            w_{3,i} & 0 \\
            0 & 0
          \end{bmatrix} +  \begin{array}{l} 
                  \redgr \\
                  \\
                  \\
                  \\
                \end{array} \begin{bmatrix}
          0 & \times  \\
          0 & d_{i+1} \\
          0 & w_{1,i+1} \\
          0 & w_{2,i+1} \\
          0 & w_{3,i+1} \\
        \end{bmatrix}  \right)}_{C} \begin{array}{l}
         \redgr \\ \\ \\ \\ \end{array} + R = 
      \ascendinggrhats
       \begin{bmatrix}
        \hat{d}_i & 0  \\
        \hat{w}_{1,i} & 0 \\
        \hat{w}_{2,i} & 0 \\
        \hat{w}_{3,i} & 0 \\
        0 & 0
      \end{bmatrix} +  \begin{bmatrix}
      0 & \times  \\
      0 & \hat{d}_{i+1} \\
      0 & \hat{w}_{1,i+1} \\
      0 & \hat{w}_{2,i+1} \\
      0 & \hat{w}_{3,i+1} \\
    \end{bmatrix}. 
  \end{equation}
  }

  Let $C$ be the left matrix in \eqref{eq:conjugation}. 
  Then  the elements $\hat{w}_{j,i+1}$ must coincide
  with the vector $C[3:,2]$, the diagonal elements
  $\hat{d}_i$ and $\hat{d}_{i+1}$ are determined by the 
  diagonal of the top $2 \times 2$ block of $C$. It remains to determine
  the elements $\hat{w}_{j,i}$ and the value $\alpha$. To find them 
  we can  multiply on the left by the inverses of the rotations
  in $\hat{\mathcal G}[i]$. We get the equation
  \[
    \descendinggrhats \left( C e_1 + \begin{bmatrix} 0 \\ \alpha \\ 0 \\ 0 \\ 0 \\ \end{bmatrix} \right) = \begin{bmatrix}
            \hat{d}_i \\
            \hat{w}_{1,i} \\
            \hat{w}_{2,i} \\
            \hat{w}_{3,i} \\
            0 \\ 
          \end{bmatrix}. 
  \]  
  We can choose $\alpha$ such that we get a $0$ in the last
  component (which can always be done if the rotations are not trivial)
  and then set the values $\hat{w}_{j,i}$ by back substitution. 
  
  The algorithm presented above allows us to find the residual matrix
  $R$ and a new $k$-quasiseparable matrix $\hat{M}$ with the desired
  rotations in its representations such that $\hat{M} + R = F_i M
  F_i^*$.  This is what we desire in order to implement the
  algorithm. The reader may wonder if $\hat{M}$ is exactly the matrix
  $t(F_i U V F_i^*)$. This is actually true by assuming some
  hypothesis about the irreducibility of the matrices involved and
  that $i$ is not bigger than $n-k-1$, essentially due to the fact
  that $F_iMF_i^*$ is forced to be equal to $t(F_iUV^*F^*)$ in all the
  elements but the $2 \times 2$ submatrix of indices $(i,i+1)$. We
  omit the analysis of this property since, in the end, it is not very
  relevant for the algorithm. The important result is just that we
  have found a matrix with the correct Givens rotations such that the
  sum of the lower parts is contained in the same span. This last
  claim clearly holds in this case.

  \section{Reduction algorithm}
    \label{sec:reduction-algorithm}
    
    In this section we explain how the reduction algorithm can be
    constructed by using the tools presented in the previous sections.
    
    Recall that the matrix $A = D + UV^*$ can be represented in the
    more general form $A = t(UW^*) + S + UV^*$, where $S$ is a
    Hermitian $1$-quasiseparable matrix and $U,V, W\in \mathbb{C}^{n
      \times k}$, just by setting $W=0$ and $S=D$. Recall also that,
    by Theorem \ref{th:0}, this form is maintained by the trailing
    principal submatrices $\hat A_j\in\mathbb
    C^{(n-j+1)\times(n-j+1)}$ of the matrices $A_j$, generated at each
    step $j$ of the algorithm, that is, $\hat
    A_j=t(U_jW_j^*)+S_j+U_jV_j^*$.  The matrices $U_j$ and $V_j$ are
    easily obtained by multiplying $U_{j-1}$ and $V_{j-1}$ by a
    sequence of Givens rotations and by removing the first row. The
    matrices $S_j$ and $W_j$ will be used to store the ``residues''.
    
    A high level overview of the algorithm is reported in the
    pseudo-code of Algorithm~\ref{alg:highlevelreduction}. 

    \begin{algorithm} 
      \caption{High level reduction process}
      \label{alg:highlevelreduction}
    \begin{algorithmic}[1]
      \State $A_1 \gets D + UV^*$
      \State $M \gets 0$
      \State $S \gets 0$
      
      \State $s \gets {\tt zeros}(1,n-1)$
      \State $d \gets {\tt zeros}(1,n)$
      
      \For{$i = 1, \dots, n - 2 - k$}
        \State $\mathcal G \gets {\tt cleanColumn}(A_{i}[:,1])$
        \State $d[i] \gets (\mathcal G A_{i}[:,1])[1]$
        \State $s[i] \gets (\mathcal G A_{i}[:,1])[2]$
        \State $U \gets (\mathcal G \cdot U) [2:n,:]$
        \State $V \gets (\mathcal G \cdot V) [2:n,:]$
        \State $(R_M,M) \gets {\tt conjugateAndTruncate}(M,\mathcal G)$
        \State $(R_S,S) \gets {\tt conjugateAndTruncate}(S,\mathcal G)$
        \State $M \gets M + S$
        \State $S \gets R_M + R_S$
      \EndFor
      
      \State $\left( d[n-1-k:n], s[n-1-k:n-1],U,V \right) \gets {\tt reduceTrailingBlock}(A_{n-1-k})$
      
    \end{algorithmic}
    \end{algorithm}
    
    The functions in the code of Algorithm~\ref{alg:highlevelreduction}
    perform the following operations: 
    
    \begin{description}
      \item[{ \tt cleanColumn}$(v)$] is a function that takes as input 
       a column vector and returns a sequence of Givens rotations $\mathcal G$ such that $\mathcal G v = v_1 e_1 + \alpha e_2$ for some $\alpha$. 
      \item[$(R_M,M) \gets ${\tt conjugateAndTruncate}$(M,\mathcal G)$] takes as 
      input a quasiseparable
       Hermitian matrix $M \in \qsh{k}$ and a sequence of Givens rotations
       $\mathcal G$. Then it computes a quasiseparable representation
       for $\mathcal G M \mathcal G^* - R_M$ where $R_M$ is a matrix
       in $\qsh{1}$. It returns an updated representation of $M$ and
        the residual matrix $R_M$. 
      \item[$S \gets R_M + R_S$] computes the sum of the matrices
       $R_M, R_S \in \qsh{1}$. This is done by assuming that both 
       have the same sequence of Givens rotations in their representation. 
      \item[{\tt reduceTrailingBlock}$(A)$] reduces the last $k \times k$
       block of the matrix using a standard Hessenberg reduction process. 
       This is done because, in the last steps, the trailing block 
       does not have any particular structure anymore. 
    \end{description}
    
    Some numerical issues might be encountered in the above version of
    the algorithm. For instance, some cancellation may happen in the
    sum $R_M + R_S$, which eventually may affect the Givens rotations
    of the representation of $M$.
    
    A technique based on re-orthogonalization can be used to restore
    better approximations. Recall that the rotations inside the GV
    representation of $M$ are such that $\mathcal G^* U = \left[ \begin{smallmatrix}
      Z \\ 0 
    \end{smallmatrix} \right]$. Such rotations are not unique but (at
    least with some hypothesis on irreducibility)
    are essentially unique, that is, they can be determined 
    up to a multiplicative constant of modulus $1$. Based on this information
    we can compute rotations in order to obtain $\mathcal G^* U$ in
    the desired form and correct the moduli of the sine and cosine
    inside $\mathcal G$ without altering the signs. 
    
    This has shown to be quite effective in practice, leading to better 
    numerical results. The cost of a reorthogonalization is the cost
    of a QR factorization of $U$, thus asymptotically $O(nk^2)$. By performing
    it every $k$ steps we have a total cost of the modified reduction
    algorithm that is still $O(n^2k)$, since we need $O(n)$ steps to
    complete the reduction and $O(nk) \cdot \frac 1 k O(n) = O(n^2k)$ ops. 
    
  \section{Numerical experiments}
    \label{sec:numerical-experiments}
    
    The algorithm presented in Section~\ref{sec:reduction-algorithm} 
    has been implemented in the Julia language. It has been run on a Laptop
    with an Intel(R) Core(TM) i3-2367M CPU running at 1.40GHz and 4
    GB of RAM. 
    
    In order to analyze the complexity and the accuracy
    of the results we have performed the following tests: 
    \begin{itemize}
      \item We have run the algorithm on matrices of different sizes but 
        with constant quasiseparability rank $k = 10$. The purpose is to verify that the CPU
        time is quadratic in the dimension of the problem. 
      \item We have run the algorithm at a fixed dimension $n = 200$
        with values of $k$ between $5$ and $160$. Here the goal is to verify that
        the CPU time grows linearly with the rank. 
      \item We have run the algorithm on some test problems in order
        to measure the errors on the eigenvalues computed starting
        from the final Hessenberg form. The purpose of this set of
        tests is to check the numerical stability of the algorithm.
    \end{itemize}
    
    Every experiment has been run $10$ times and the mean value
    of the timings has been taken.
    In Figure~\ref{fig:exp_n}
    we have reported, in log scale, the timings for some 
    experiments with $n = 100i$ for $i = 1, \dots, 10$. 
    In Figure~\ref{fig:exp_k} we have
    reported the CPU time in the case of
    matrices of fixed size $n = 400$ 
    with various quasiseparable ranks ranging from $5$ to 
    $160$.

    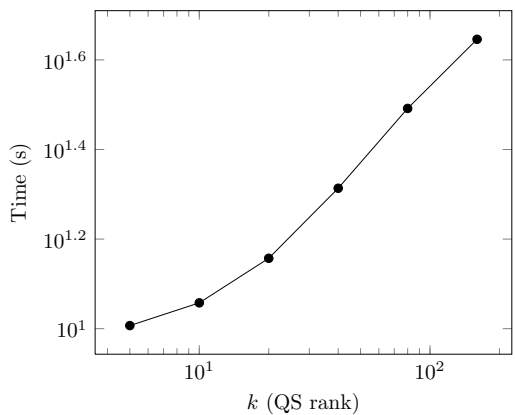
\begin{figure}
     \begin{minipage}{0.6\linewidth}
      \begin{tikzpicture}[scale=.8]
      \begin{loglogaxis}[xlabel=$n$ (Size),
          ylabel=Time (s),
          xmin=8e1,
          xmax=1.2e3]
      	\addplot[mark=*] table {exp_n.dat};
        \addplot[red,mark=none] table {interp_exp_k.dat};
      \end{loglogaxis}
      \end{tikzpicture}
      \end{minipage}
      \begin{minipage}{0.25\linewidth}
        \begin{tabular}{ll}
          Size & Time (s) \\ \hline \hline
          100 & 0.65 \\
          200 & 2.68 \\
          300 & 6.06 \\
          400 & 11.05 \\
          500 & 17.5 \\
          600 & 25.8 \\
          700 & 35.7 \\
          800 & 49.4 \\
          900 & 67.17 \\
          1000 & 77.56 \\
        \end{tabular}
      \end{minipage}
      \caption{CPU time, in seconds, for the Hessenberg reduction of a diagonal
      plus rank $10$ matrix of size $n$. Here the line is the plot
      of $\gamma n^2$ for an appropriate $\gamma$. It is evident the quadratic behavior of the time.}
      \label{fig:exp_n}
    \end{figure}

    \begin{figure} 
    \begin{minipage}{0.6\linewidth}
      \begin{tikzpicture}[scale=.8]
      \begin{loglogaxis}[xlabel=$k$ (QS rank),ylabel=Time (s)]
      	\addplot[mark=*] table {exp_k.dat};
      \end{loglogaxis}
      \end{tikzpicture}
      \end{minipage}
      \begin{minipage}{0.35\linewidth}
        \begin{tabular}{ll}
          QS Rank & Time (s) \\\hline \hline
          5 & 10.16 \\
          10& 11.42 \\
          20& 14.36 \\
          40& 20.58 \\
          80& 31.02 \\
          160& 44.25 \\
        \end{tabular}      
      \end{minipage}
      \caption{CPU time, in seconds, for the Hessenberg reduction of a $400 \times 400$ diagonal plus rank $k$ matrix.}\label{fig:exp_k}
    \end{figure}
    
    Looking at the results in Figure~\ref{fig:exp_k} we see that the
    complexity in the rank is almost sublinear at the start. This is
    due to the inefficiency of operations on small matrices and the
    overhead of these operations in our Julia implementation. The
    linear trend starts to appear for larger ranks.

            \begin{figure} 
            \begin{minipage}{0.47\linewidth}
              \begin{tikzpicture}[scale=.7]
              \begin{loglogaxis}[xlabel=$n$ (Size),
                    ylabel=Absolute error,
                    ymax=1e-2
                ]
              	\addplot[mark=*] table {err.dat};
              	\addplot[mark=square] table {err_max.dat};
              	\addplot[mark=x] table {err_min.dat};
              	\legend{Mean,Maximum,Minimum};
              \end{loglogaxis}
              \end{tikzpicture}
              \end{minipage}
              \begin{minipage}{0.52\linewidth}\small
    			\begin{tabular}{cccc} 
    			  Size & Mean & Minimum & Maximum \\ \hline \hline
    			     40.0  & 5.52e-14  & 1.11e-15  & 3.97e-13 \\
    			     80.0  & 2.59e-13  & 0.0          & 2.43e-12 \\
    			    160.0  & 5.23e-13  & 5.32e-15  & 3.74e-12 \\
    			    320.0  & 5.06e-12  & 1.49e-14  & 2.41e-10 \\
    			    640.0  & 1.80e-10  & 1.42e-13  & 2.43e-9  \\
    			   1280.0  & 8.43e-9   & 6.43e-15  & 3.32-7   \\
    		  \end{tabular}
              \end{minipage}
              \caption{Absolute errors on eigenvalues computation for random matrices of quasiseparable rank $30$ and variable sizes. }\label{fig:exp_err}
            \end{figure}
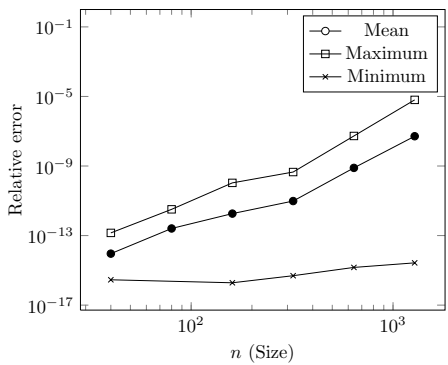
            
                    \begin{figure} 
                    \begin{minipage}{0.47\linewidth} 
                      \begin{tikzpicture}[scale=.7]
                      \begin{loglogaxis}[xlabel=$n$ (Size),
                            ylabel=Relative error,
                            ymax=1
                        ]
                      	\addplot[mark=*] table {rel_err.dat};
                      	\addplot[mark=square] table {rel_err_max.dat};
                      	\addplot[mark=x] table {rel_err_min.dat};
                      	\legend{Mean,Maximum,Minimum};
                      \end{loglogaxis}
                      \end{tikzpicture}
                      \end{minipage}
                      \begin{minipage}{0.52\linewidth} \small
            			\begin{tabular}{cccc}
            			  Size & Mean & Minimum & Maximum \\ \hline \hline
            		   40.0 & 9.13e-15 & 2.86e-16 & 1.42e-13 \\
            		   80.0 & 2.55e-13 & 0.0         & 3.22e-12 \\
            		  160.0 & 1.83e-12 & 1.92e-16 & 1.07e-10 \\
            		  320.0 & 9.66e-12 & 4.91e-16 & 4.49e-10 \\
            		  640.0 & 7.79e-10 & 1.47e-15 & 5.34e-8  \\
            		 1280.0 & 5.15e-8  & 2.69e-15 & 6.35e-6  \\  
            		  \end{tabular}
                      \end{minipage}
                      \caption{Relative errors on eigenvalues computation for random matrices of quasiseparable rank $30$ and variable sizes. }\label{fig:exp_rel_err}
                    \end{figure}
            
         As a last experiment in Figure~\ref{fig:exp_err} and Figure~\ref{fig:exp_rel_err} we have reported
         the absolute and relative errors, respectively, on eigenvalue computations for various sizes
         and fixed quasiseparable rank. The errors were obtained as differences between the eigenvalues computed from the starting full matrix using
         the QR algorithm and the QR algorithm applied to the Hessenberg matrix provided
         by our algorithm. 
         
         The matrices in these examples have been obtained by using the
         {\tt randn} function that constructs matrices whose elements are drawn
         from a $N(0,1)$ Gaussian distribution. This function has been used
         to construct $D$, $U$ and $V$ diagonal and $n \times k$, respectively, such that $A = D + UV^*$. 

\section{An application}\label{sec:appl}
Let $P(x)=\sum_{i=0}^d P_ix^i$ be a matrix polynomial where $P_i$ are
$k\times k$ matrices. In \cite{secular-linearization} a companion
linearization $A$ for $P(x)$ has been introduced where $A$ is an
$n\times n$ matrix, $n=dk$, of the form \eqref{eq:d+lr} with $n=dk$.  The
computation of the eigenvalues of $P(x)$, that is, the solutions of
the equation $\det P(x)=0$, is therefore reduced to solving the linear
eigenvalue problem for the matrix $A$. The availability of the
Hessenberg form $H$ of $A$ with the quasiseparability structure,
enables one to apply the QR iteration to $H$ at a low cost by exploiting
the both the Hessenberg and the quasiseparable structures.

A different approach to solve the equation $\det P(x)=0$, followed in
\cite{noferini}, consists in applying the Ehrlich-Aberth iteration to the
polynomial $\det P(x)$, or alternatively, to represent the polynomial
$P(x)$ in secular form \cite{br:jcam}. In this case, one has to compute
$\det P(x)$ at different values $x_1,\ldots, x_n$. Indeed, the
evaluation of $\det P(x)$ at a single value $x=\xi$ can be performed
by applying the Horner rule in order to compute the matrix $P(\xi)$ and
then by applying Gaussian elimination for computing the determinant
$\det P(\xi)$. The overall cost is $O(dk^2+k^3)$ so that the
computation at $n=dk$ different points has the cost $O(d^2k^3+dk^4)$.

The availability of the structured Hessenberg form $H$ allows us to reduce
this cost. We will show that computing $\det(xI-A)=\det(xI-H)$ can be performed
in $O(nk)=O(dk^2)$ operations. Asymptotically, the value of this cost is the minimum
that we can obtain. In fact, the matrix $A$ is defined by $2nk+n$
entries, so that any algorithm which takes in input these $2nk+n$
values must perform at least $nk+n/2$ operations. In fact each pair of
input data must be involved in at least an operation.  In the case where we
have to compute $\det (xI-A)$ at $dk$ different values we obtain the
cost $O(d^2k^3)$ which improves the cost required by the Horner rule applied to the matrix polynomial.

The algorithm for computing $\det (xI-H)$ at a low cost relies on the
Hyman method \cite{wilkinson}.  Assume that $H$ is in
$k$-quasiseparable Hessenberg form.  Consider the system
$(xI-H)v=\alpha e_1$, set $v_n=1$ so that the equations from the
second to the last one form a triangular $k$-quasiseparable
system. The first equation provides the value of $\alpha$ after that
$v_1,\ldots,v_{n-1}$ have been computed.  Since a triangular
quasiseparable system can be solved with $O(nk)$ ops, we are able to
compute $\alpha$ at the same cost. On the other hand, by the Cramer
rule we find that $1=v_n=\alpha(\prod_{i=2,n}h_{i,i-1})/\det(xI-H)$
which provides the sought value of $\det(xI-H)$. A similar approach
can be used for computing the Newton correction $p(x)/p'(x)$ for
$p(x)=\det(xI-H)$.


  \bibliographystyle{siam}

  \bibliography{biblio}

\end{document}